\documentclass[12pt]{amsart}
\usepackage{amsfonts, amssymb, latexsym, epsfig, amsthm}
\usepackage[usenames,dvipsnames]{xcolor}
\usepackage{graphicx,pgf,tikz}
\usetikzlibrary{decorations.markings}
\usepackage{ytableau}
\usepackage[all]{xy}
\usepackage{wrapfig}

\newtheorem{theorem}{Theorem}[section]

\newtheorem{lemma}[theorem]{Lemma}

\newtheorem{claim}[theorem]{Claim}

\newtheorem*{claim*}{Claim}
\newtheorem{corollary}[theorem]{Corollary}
\newtheorem{Main Conjecture}[theorem]{Main Conjecture}

\theoremstyle{remark}

\theoremstyle{plain}

\usepackage{enumitem}

\begin{document}
\pagestyle{plain}

\title{Schubert polynomials, 132-patterns, and Stanley's conjecture}
\author{Anna Weigandt}
\address{Dept.~of Mathematics, U.~Illinois at
Urbana-Champaign, Urbana, IL 61801, USA}
\email{weigndt2@uiuc.edu}

\date{\today}

%\begin{abstract}
%Motivated by a recent conjecture of R.~P.~Stanley we offer
%a lower bound for the sum of the coefficients of
%a Schubert polynomial in terms
%of $132$-pattern containment.
%\end{abstract}

\maketitle

%\tableofcontents

\section{Introduction}

This paper is motivated by a conjecture
of  R.~P.~Stanley \cite[Conj.~4.1]{stanley2017some} concerning the \emph{Schubert polynomials} of  A.~Lascoux and M.-P.~Sch\"utzenberger \cite{Lascoux.Schutzenberger}. 
   If $w_0=n \, n-1 \, \ldots \, 1$ is the longest permutation in $S_n$ then \[\mathfrak S_{w_0}:=x_1^{n-1}x_2^{n-2}\cdots x_{n-1}.\]   For any other $w\in S_n$, there is some $i$ so that $w(i)<w(i+1)$.  Then $\mathfrak S_{w}=\partial_i \mathfrak S_{ws_i}$, where $\displaystyle\partial_if:=\frac{f-s_if}{x_i-x_{i+1}}$ and $s_i=(i,i+1)$ acts on $f$ by exchanging the variables $x_i$ and $x_{i+1}$. The $\partial_i$'s satisfy the same braid and commutativity relations as the simple transpositions and so $\mathfrak S_w$ is well defined.

We are interested in the following specialization: $\nu_w:=\mathfrak S_w(1,1,\ldots, 1).$  Let
\begin{equation}
\label{eqn:132patterns}
 P_{132}(w):=\{(i,j,k): i<j<k  \text{ and } w(i)<w(k)<w(j)\}.
\end{equation}
Write $\eta_w:=\#P_{132}(w)$.
If $\eta_w\geq 1$ then $w$ {\bf contains} the pattern $132$.
We prove that $\eta_w$ provides a lower bound for $\nu_w$.
\begin{theorem}[The $132$-bound]
\label{theorem:bound}
For any $w\in S_n$, $\nu_w\geq \eta_w+1$.
\end{theorem} 
As a corollary, we obtain the following conjecture of R.~P.~Stanley \cite[Conj.~4.1]{stanley2017some}.
\begin{corollary}
\label{cor:main}
$\nu_w=2$ if and only if $\eta_w=1$.
\end{corollary}

\begin{proof}

Let $w\in S_n$. If $\eta_w=0$ then $\nu_w=1$ \cite[Chapter~4]{macdonald1991notes}. If $\eta_w=1$ then $\nu_w=2$ \cite[Section~4]{stanley2017some}.  
Otherwise, $\eta_w\geq 2$.  Then we apply Theorem~\ref{theorem:bound} and obtain
\[\nu_w\geq \eta_w+1 \geq 3.\]
As such, $\nu_w=2$ if and only if $\eta_w=1$.
\end{proof}

\section{Background on Permutations and Pipe Dreams}
\label{section:background}

We will recall the necessary background on permutations and Schubert polynomials; our references are \cite[Ch.~2]{manivel2001symmetric} and \cite{bergeron1993rc} respectively.  The {\bf Rothe diagram} of $w\in S_n$ is the set 
\begin{equation}
\label{eqn:rothe}
D(w):=\{(i,j):1\leq i,j\leq n, w(i)>j, \text{ and } w^{-1}(j)>i\}.
\end{equation} 
Notice immediately from (\ref{eqn:rothe}), we have 
\begin{equation}
\label{eqn:diagramsym}
D(w^{-1})=D(w)^t.
\end{equation}

\begin{wrapfigure}{r}{0.2\textwidth}
\centering
 \begin{tikzpicture}[x=1em,y=1em]
      \draw[color=black, thick](0,1)rectangle(7,8);
     \filldraw[color=black, fill=gray!30, thick](0,7)rectangle(1,8);
     \filldraw[color=black, fill=gray!30, thick](1,7)rectangle(2,8);
     \filldraw[color=black, fill=gray!30, thick](2,7)rectangle(3,8);
     \filldraw[color=black, fill=gray!30, thick](0,6)rectangle(1,7);
     \filldraw[color=black, fill=gray!30, thick](1,6)rectangle(2,7);
     \filldraw[color=black, fill=gray!30, thick](4,6)rectangle(5,7);
     \filldraw[color=black, fill=gray!30, thick](2,6)rectangle(3,7);
     \filldraw[color=black, fill=gray!30, thick](5,6)rectangle(6,7);
     \filldraw[color=black, fill=gray!30, thick](0,5)rectangle(1,6);
     \filldraw[color=black, fill=gray!30, thick](2,3)rectangle(3,4);
     \filldraw[color=black, fill=gray!30, thick](4,3)rectangle(5,4);
     \filldraw [black](3.5,7.5)circle(.1);
     \filldraw [black](6.5,6.5)circle(.1);
     \filldraw [black](1.5,5.5)circle(.1);
     \filldraw [black](.5,4.5)circle(.1);
     \filldraw [black](5.5,3.5)circle(.1);
     \filldraw [black](2.5,2.5)circle(.1);
     \filldraw [black](4.5,1.5)circle(.1);
     \draw[thick] (3.5,7.5)--(3.5,7);
     \draw[thick] (4,7.5)--(3.5,7.5);
     \draw[thick] (6.5,6.5)--(6.5,6);
     \draw[thick] (7,6.5)--(6.5,6.5);
     \draw[thick] (1.5,5.5)--(1.5,5);
     \draw[thick] (2,5.5)--(1.5,5.5);
     \draw[thick] (.5,4.5)--(.5,4);
     \draw[thick] (1,4.5)--(.5,4.5);
     \draw[thick] (5.5,3.5)--(5.5,3);
     \draw[thick] (6,3.5)--(5.5,3.5);
     \draw[thick] (2.5,2.5)--(2.5,2);
     \draw[thick] (3,2.5)--(2.5,2.5);
     \draw[thick] (4.5,1.5)--(4.5,1);
     \draw[thick] (5,1.5)--(4.5,1.5);
     \draw[thick] (5,7.5)--(4,7.5);
     \draw[thick] (6,7.5)--(5,7.5);
     \draw[thick] (7,7.5)--(6,7.5);
     \draw[thick] (3,5.5)--(2,5.5);
     \draw[thick] (4,5.5)--(3,5.5);
     \draw[thick] (5,5.5)--(4,5.5);
     \draw[thick] (6,5.5)--(5,5.5);
     \draw[thick] (7,5.5)--(6,5.5);
     \draw[thick] (2,4.5)--(1,4.5);
     \draw[thick] (3,4.5)--(2,4.5);
     \draw[thick] (4,4.5)--(3,4.5);
     \draw[thick] (5,4.5)--(4,4.5);
     \draw[thick] (6,4.5)--(5,4.5);
     \draw[thick] (7,4.5)--(6,4.5);
     \draw[thick] (7,3.5)--(6,3.5);
     \draw[thick] (4,2.5)--(3,2.5);
     \draw[thick] (5,2.5)--(4,2.5);
     \draw[thick] (6,2.5)--(5,2.5);
     \draw[thick] (7,2.5)--(6,2.5);
     \draw[thick] (6,1.5)--(5,1.5);
     \draw[thick] (7,1.5)--(6,1.5);
     \draw[thick] (3.5,7)--(3.5,6);
     \draw[thick] (3.5,6)--(3.5,5);
     \draw[thick] (6.5,6)--(6.5,5);
     \draw[thick] (1.5,5)--(1.5,4);
     \draw[thick] (3.5,5)--(3.5,4);
     \draw[thick] (6.5,5)--(6.5,4);
     \draw[thick] (.5,4)--(.5,3);
     \draw[thick] (1.5,4)--(1.5,3);
     \draw[thick] (3.5,4)--(3.5,3);
     \draw[thick] (6.5,4)--(6.5,3);
     \draw[thick] (.5,3)--(.5,2);
     \draw[thick] (1.5,3)--(1.5,2);
     \draw[thick] (3.5,3)--(3.5,2);
     \draw[thick] (5.5,3)--(5.5,2);
     \draw[thick] (6.5,3)--(6.5,2);
     \draw[thick] (.5,2)--(.5,1);
     \draw[thick] (1.5,2)--(1.5,1);
     \draw[thick] (2.5,2)--(2.5,1);
     \draw[thick] (3.5,2)--(3.5,1);
     \draw[thick] (5.5,2)--(5.5,1);
     \draw[thick] (6.5,2)--(6.5,1);
     \end{tikzpicture}
\end{wrapfigure} 

We may  visualize $D(w)$ as follows.
For each $i=1,\ldots, n$,  plot $(i,w(i))$.  Then,  strike out all boxes to the right and below each of the plotted points.  The boxes which remain form $D(w)$.  
For example, $D(4721635)$ is pictured to the right.
The {\bf length} of a permutation is the number of boxes in its diagram,
$\ell(w):=\#D(w).$
Each permutation has an associated rank function $r_w$, where
\begin{equation}
\label{eqn:rank}
r_w(i,j):=\#\{k:1\leq k\leq i \text{ and } w(k)\leq j\}.
\end{equation}

Schubert polynomials can be written as a sum over \emph{pipe dreams}.   Pipe dreams appear in the literature under various names; they are the \emph{pseudo-line configurations}  of  S.~Fomin and A.~N.~Kirillov  \cite{fomin1996yang}  and  the  \emph{RC-graphs} of N.~Bergeron and S.~C.~Billey \cite{bergeron1993rc}.   They were studied from a geometric perspective by A.~Knuston and E.~Miller \cite{knutson2005grobner}. 

Let $\mathbb Z_{> 0}\times \mathbb Z_{> 0}$ be the semi-infinite grid, starting from the northwest corner.  Label the rows and columns in matrix notation, i.e. position $(i,j)$ indicates the $i$th row from the top and the $j$th column from the left.  An {\bf pipe dream} is a tiling of this grid with $+$'s (pluses) and \includegraphics[scale=.6]{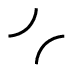}'s (elbows).    For simplicity, we will sometimes draw the elbows as dots.

We  freely identify each pipe dream with a subset of  $\mathbb Z_{> 0}\times  \mathbb Z_{> 0}$ by recording the coordinates of the pluses.
Associate a weight monomial to $\mathcal P$:
\[{\tt wt}(\mathcal P)=\prod_{(i,j)\in \mathcal P} x_i.\]
Equivalently, the exponent of $x_i$ counts the number of pluses which appear in row $i$ of $\mathcal P$.

\begin{wrapfigure}{l}{0.17\textwidth}
\centering
\vspace{-12pt}
\includegraphics[width=0.17\textwidth]{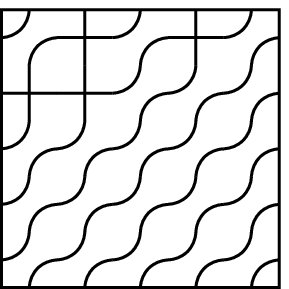}
\vspace{-25pt}
\end{wrapfigure} 

We may interpret $\mathcal P$ as a collection of overlapping strands, using the rule that a strand never bends at a right angle.  The $+$'s indicate the positions where two strands cross.  
Each  row on  the left edge of $\mathbb Z_{> 0}\times \mathbb Z_{> 0}$ is connected by some strand to a unique column along the top, and vice versa.  
If the $i$th row is connected to the $j$th column, let $w_{\mathcal P}(i):=j$.  There exists some $n$ so that $w_{\mathcal P}(i)=i$ for all $i>n$, so $w_{\mathcal P}\in S_\infty$.  In practice, we identify $w_\mathcal P$ with its representative in some finite symmetric group. For example, if $\mathcal P$ is the pipe dream pictured above, then we write $w_\mathcal P=15324$.

If $\#\mathcal P=\ell(w_\mathcal P)$ then $\mathcal P$ is {\bf reduced}.  
Let \[{\tt RP}(w):=\{\mathcal P: w_\mathcal P=w \text{ and } \mathcal P \text { is reduced}\}.\]  
\begin{theorem}[\cite{bergeron1993rc,fomin1996yang}]
\begin{equation}
\label{eqn:schubertdef}
\mathfrak S_w=\sum_{\mathcal P\in {\tt RP}(w)}{\tt wt}(\mathcal P).
\end{equation} 
\end{theorem}
Recall, $\nu_w:=\mathfrak S_w(1,1,\ldots, 1).$  Immediately from (\ref{eqn:schubertdef}), $\nu_w=\#{\tt RP}(w)$.

There are two pipe dreams which have an explicit description in terms of $w$.  Let 
\begin{equation}
\label{eqn:m}
m_i(w)=\#\{j:j>i \text{ and } w(j)<w(i)\}.
\end{equation} 
Then the {\bf bottom pipe dream} is
\begin{equation}
\label{eqn:botbound}
\mathcal B_w=\{(i,j):j\leq m_i(w)\}.
\end{equation} 
Graphically, $\mathcal B_w$ is obtained from $D(w)$ by replacing each box with a plus and then left justifying within each row. 
We define the {\bf top pipe dream} as the transpose of the bottom pipe dream of $w^{-1}$:  
\[\mathcal T_w:=\mathcal B_{w^{-1}}^t.\]  By (\ref{eqn:diagramsym}), $\mathcal T_w$ is obtained from $D(w)$ by top justifying pluses within columns.

N.~Bergeron and S.~C.~Billey gave a procedure to obtain any pipe dream in ${\tt RP}(w)$ algorithmically, starting from $\mathcal B_w$.
A {\bf ladder move} is an operation on pipe dreams which produces a new pipe dream by a replacement of the following type:
\[\begin{array}{cc}
\cdot &\cdot \\
+&+\\
+&+\\
\vdots &\vdots \\
+&+\\
 + & \cdot 
\end{array} \quad \mapsto \quad 
\begin{array}{cc}
\cdot &+ \\
+&+\\
+&+\\
\vdots &\vdots \\
+&+\\
 \cdot & \cdot 
\end{array} \]
In the above picture, the columns and rows are consecutive.  
 If $\mathcal P\mapsto \mathcal P'$ is a ladder move, then $\mathcal P\in {\tt RP}(w)$ if and only if $\mathcal P'\in {\tt RP}(w)$ .  In other words,  ${\tt RP}(w)$ is closed under ladder moves \cite{bergeron1993rc}.   Furthermore, ${\tt RP}(w)$ is connected by ladder moves.
\begin{theorem}[Theorem 3.7 \cite{bergeron1993rc}]
\label{theorem:bb}
If $\mathcal P\in {\tt RP}(w)$, then $\mathcal P$ can be obtained by a sequence of ladder moves from $ \mathcal B_w$.
\end{theorem}
We will mostly focus on a special type of ladder move.  A {\bf simple ladder move} is a replacement of the following form:
\[\begin{array}{cc}\cdot &\cdot \\ + & \cdot \end{array} \quad \mapsto \quad \begin{array}{cc}\cdot &+ \\ \cdot & \cdot \end{array} \]
 In Lemma~\ref{lemma:biject}, we show that any sequence of ladder moves connecting $\mathcal B_w$ to $\mathcal T_w$ must contain only simple ladder moves.  We use this special structure to count the exact number of pipe dreams in any such sequence, providing a lower bound for ${\tt RP}(w)$.

\section{Proof Theorem~\ref{theorem:bound}}
\label{section:proof}

We start by interpreting $\eta_w$ as a weighted sum over $D(w)$.
\begin{lemma}
\label{lemma:eta}
\[\eta_w=\sum_{(i,j)\in D(w)} r_w(i,j)\]
\end{lemma}
\begin{proof}
Suppose $(i,j,k)\in  P_{132}(w)$.  Then $w(j)>w(k)$ and $w^{-1}(w(k))=k>j$.  By  (\ref{eqn:rothe}), we have $(j,w(k))\in D(w)$.  Furthermore, $i\leq j$ and $w(i)\leq w(k)$.  Then by (\ref{eqn:rank}), 
\[\#\{\ell:(\ell,j,k)\in  P_{132}(w)\}\leq \#\{\ell:\ell\leq j \text{ and } w(\ell)\leq w(k)\}=r_w(j,w(k)) .\]
Then
\begin{equation}
\label{eqn:1}
\eta_w\leq \sum_{(i,j)\in D(w)} r_w(i,j).
\end{equation}
On the other hand, suppose $(i,j)\in D(w)$.  Then \[w(i)>j=w(w^{-1}(j)) \text{ and } w^{-1}(j)>i.\]  Take \[k\in \{k:k\leq i \text{ and } w(k)\leq j\}.\]   Since $(i,j)\in D(w)$, we must have $k<i$ and $w(k)<j$.  Then \[k<i<w^{-1}(j) \text{ and } w(k)<w(w^{-1}(j))<w(i)\] and so 
\[(k,i,w^{-1}(j))\in  P_{132}(w).\]
So if $(i,j)\in D(w)$, \[ \#\{\ell:(\ell,i,w^{-1}(j))\in  P_{132}(w)\}\geq r_w(i,j).\]
Therefore, 
\begin{equation}
\label{eqn:2}
\eta_w\geq \sum_{(i,j)\in D(w)} r_w(i,j).
\end{equation}
Then combining (\ref{eqn:1}) and (\ref{eqn:2}) gives
\[\eta_w= \sum_{(i,j)\in D(w)} r_w(i,j).\qedhere \]
\end{proof}

If $\mathcal P\in{\tt RP}(w)$, let $\mathbf a_{\mathcal P}:=(a_\mathcal P(1),\ldots, a_\mathcal P(n))$ where 
\begin{equation}
\label{eqn:antidiagcount}
a_\mathcal P(k)=\#\{(i,j)\in \mathcal P:i+j-1=k\}.
\end{equation}
 Equivalently, $a_\mathcal P(k)$ is the number of pluses that occur in the $k$th antidiagonal of $\mathcal P$.  

\begin{lemma}
\label{lemma:simpleladder}
Suppose there is a path of ladder moves from  $\mathcal P$ to $\mathcal Q$:
\begin{equation}
\label{eqn:otherplus}
\mathcal P=\mathcal P_0\mapsto\mathcal P_1\mapsto \ldots \mapsto \mathcal P_N=\mathcal Q.
\end{equation}
   Each ladder move in (\ref{eqn:otherplus}) is simple if and only if $\mathbf a_{\mathcal P}=\mathbf a_{\mathcal Q}$.
\end{lemma}

\begin{proof}
\noindent $(\Rightarrow)$ Assume each $\mathcal P_i\mapsto \mathcal P_{i+1}$ is a simple ladder move. Then $\mathcal P_{i+1}$ is obtained from $\mathcal P_{i}$ by moving a single plus to a new position in the same antidiagonal.  So  $\mathbf a_{\mathcal P_i}=\mathbf a_{\mathcal P_{i+1}}$ for each $i$.  Therefore $\mathbf a_{\mathcal P}=\mathbf a_{\mathcal Q}.$

\noindent $(\Leftarrow)$
We prove the contrapositive.  Suppose there is a nonsimple ladder move in the sequence (\ref{eqn:otherplus}).  It acts by removing a plus from the $i$th antidiagonal and replacing it in the $j$th antidiagonal  with $i<j$.   In particular, we may pick $j$ to be the maximum such label. By the maximality, no plus moves into the $j$th antidiagonal from a different antidiagonal. Then $a_{\mathcal P}(j)> a_{\mathcal Q}(j)$ and so  $\mathbf a_{\mathcal P}\neq \mathbf a_{\mathcal Q}.$
\end{proof}

Fix an indexing set $I$. A labeling of a pipe dream is an injective map $\mathcal L_\mathcal P:\mathcal P\rightarrow I$.  Suppose $\mathcal P\mapsto \mathcal P'$ is a simple ladder move.  Then $\mathcal P'$ inherits a labeling from $\mathcal P$ as follows:
\[\mathcal L_{\mathcal P'}(i,j)=
\begin{cases}
\mathcal L_{\mathcal P}(i,j) & \text{ if } (i,j)\in \mathcal P\\
\mathcal L_{\mathcal P}(i+1,j-1) & \text{ otherwise}.
\end{cases}\]
Since $\mathcal P\mapsto \mathcal P'$ is a simple ladder move, $\mathcal P'$ is obtained from $\mathcal P$ by adding some  $(i,j)$ to $\mathcal P$ and removing $(i+1,j-1)$.  So $\mathcal L_{\mathcal P'}$ is well defined.
If there is a path of simple ladder moves from $\mathcal P$ to $\mathcal Q$, then $\mathcal Q$ inherits the labeling $\mathcal L_\mathcal Q$ from $\mathcal L_\mathcal P$ inductively.
\begin{lemma}
\label{lemma:justsimple}
Let $L_\mathcal P$ be a labeling.  Suppose $\mathcal Q$  can be reached from $\mathcal P$ by simple ladder moves.  Then $\mathcal Q$ inherits the same labeling from $\mathcal P$ regardless of the choice of sequence.
\end{lemma}
\begin{proof}
Suppose $\mathcal P\mapsto \mathcal P'$ is a simple ladder move. Then within any antidiagonal, both pipe dreams have the same set of labels in the same relative order.  Iterate this argument along a path of simple ladder moves from $\mathcal P$ to $\mathcal Q$.  Then, in each antidiagonal, $\mathcal P$ and $\mathcal Q$ have the same set of labels, still in the same relative order.  As such, the labeling is uniquely determined and independent of the choice of path.
\end{proof}

\begin{lemma}
\label{lemma:biject}
\begin{enumerate}[label=(\Roman*)]
\item The map \[(i,j)\mapsto (i,j-r_w(i,j))\] is a bijection between $D(w)$ and $\mathcal B_w$.
\item  The map \[(i,j)\mapsto (i-r_w(i,j),j)\] is a bijection between  $D(w)$ and $\mathcal T_w$.
\item $\mathcal B_w$ and $\mathcal T_w$ are connected by simple ladder moves.
\end{enumerate}
\end{lemma}
\begin{proof}
\noindent (I) Suppose $\ell>i$ and $w(\ell)<w(i)$.  Then since $w^{-1}(w(\ell))=\ell>i$ and $w(i)>w(\ell)$, by  (\ref{eqn:rothe}), we have $(i,w(\ell))\in D(w)$.  So \[w(\ell)\in \{j:(i,w(j))\in D(w)\}.\]

If $(i,\ell)\in D(w)$, then $w(i)>\ell=w(w^{-1}(\ell))$ and $w^{-1}(\ell)>i$.  Then \[w^{-1}(\ell)\in \{j:j>i \text{ and } w(j)<w(i)\}.\]  Therefore, the two sets are in bijection and
\begin{align*}
\#\{j:(i,j)\in D(w)\} =\#\{j:j>i \text{ and } w(j)<w(i)\}=m_i(w).
\end{align*}
Then the $i$th row of $D(w)$ has as many boxes as there are pluses in the $i$th row of $\mathcal B_w$.  

Let $j_1<j_2<\ldots<j_{m_i(w)}$ be sequence obtained by sorting the set $\{j:(i,j)\in D(w)\}$. 
\begin{align*}
j_\ell-r_w(i,j_\ell)&=j_\ell-\#\{k:k\leq i \text{ and } w(k)\leq j_\ell\}\\
&=\#\{k:k> i \text{ and } w(k)\leq j_\ell\}\\
&=\#\{j:(i,j)\in D(w) \text{ and } j\leq j_\ell \}\\
&=\ell.
\end{align*}
Therefore $(i,j_\ell)\mapsto (i,\ell)$.  Since $1\leq \ell\leq m_i(w)$ the map is well defined. This holds for any $\ell\in \{1,\ldots,m_i(w)\}$ so the map is surjective.  By definition, $j_\ell=j_{\ell'}$ if and only if $\ell=\ell'$, giving injectivity.       So this is a bijection.

\noindent (II) 
Let $\phi$ be the map defined by $(i,j)\mapsto (j,i)$.  Restricted to $D(w)$, $\phi$ is a bijection between  $D(w)$ and $D(w^{-1})$. By the definition of $\mathcal T_w$, the restriction \[\phi: \mathcal B_{w^{-1}}\rightarrow \mathcal T_w \] is also a bijection.

 Let $\psi:\mathcal P(w^{-1})\rightarrow \mathcal B_w$ the map in (I).  Then the composition
\[D(w)\xrightarrow{\phi} D(w^{-1})\xrightarrow{\psi} \mathcal B_{w^{-1}} \xrightarrow{\phi}\mathcal T_w\] is a bijection.
Computing directly,
\begin{align*}
\phi(\psi(\phi(i,j)))&=\phi(\psi(j,i))\\
&=\phi(j,i-r_{w^{-1}}(j,i))\\
&=(i-r_{w^{-1}}(j,i),j).
\end{align*}
Applying (\ref{eqn:rank}),
\begin{align*}
r_{w^{-1}}(j,i)&=\#\{k:k\leq j \text{ and } w^{-1}(k)\leq i\}\\
&=\#\{\ell:w(\ell)\leq j \text{ and } w^{-1}(w(\ell))\leq i\}\\
&=\#\{\ell: \ell\leq i \text{ and } w(\ell)\leq j \}\\
&=r_w(i,j)
\end{align*}
So $\phi(\psi(\phi(i,j)))=(i-r_w(i,j),j)$.

\noindent (III)  By Theorem~\ref{theorem:bb}, there is a path of ladder moves from $\mathcal B_w$ to $\mathcal T_w$.
Applying (\ref{eqn:antidiagcount}) and the bijections in parts (I) and (II),
\begin{align*}
\mathbf a_{\mathcal B_w}(k)&=\#\{(i,j)\in D(w):i+(j-r_w(i,j))-1=k\}\\
&=\#\{(i,j)\in D(w):(i-r_w(i,j))+j-1=k\}\\
&=\mathbf a_{\mathcal T_w}(k).
\end{align*}
By Lemma~\ref{lemma:simpleladder}, the path uses only simple ladder moves.
\end{proof}
  
In light of the previous lemma, we may label the pluses of $\mathcal B_w$ using the map $(i,j)\mapsto (i,j-r_w(i,j))$, i.e. we refer to the plus which is the image of $(i,j)$ as $+_{(i,j)}$.  Likewise we label $\mathcal T_w$ using the map $(i,j)\mapsto (i-r_w(i,j),j)$.
\begin{lemma}
\label{lemma:inherit}
The above labeling of $\mathcal T_w$ is the same as the labeling it inherits from $\mathcal B_w$.
\end{lemma}
\begin{proof}
It is enough to show that within any given antidiagonal the labels in $\mathcal B_w$ and $\mathcal T_w$ are the same and have the same relative order.  If $(i,j)\in D(w)$, then $+_{(i,j)}$ is in position $(i,j-r_w(i,j))$ in $\mathcal B_w$ and in position $(i-r_w(i,j),j)$ in $\mathcal T_w$.  Since $i+j-r_w(i,j)=i-r_w(i,j)+j$, they are in the same antidiagonal.  

Now consider the $r$th antidiagonal in  $\mathcal B_w$.  Suppose the sorted list of pluses from top to bottom is \[+_{(i_1,j_1)}, +_{(i_2,j_2)}, \ldots +_{(i_k,j_k)}.\]  Since the map from $D(w)$ is by left justification, we must have $i_1<i_2<\ldots<i_k$.  Since $i_\ell+j_\ell-1=r$ for all $\ell$, it follows that $j_1>j_2>\ldots>j_k$.  Since the map from $D(w)$ to $\mathcal T_w$ is by top justification, the sorted list of pluses from top to bottom must also be 
 \[+_{(i_1,j_1)}, +_{(i_2,j_2)}, \ldots +_{(i_k,j_k)}.\]
So the labeling of $\mathcal T_w$ inherits from $\mathcal B_w$ coincides with the labeling determined by the map $(i,j)\mapsto (i-r_w(i,j),j)$.
\end{proof}

We conclude with the proof of the $132$-bound.
\begin{proof}[Proof of Theorem~\ref{theorem:bound}]
By Lemma~\ref{lemma:biject}, there is a path of simple ladder moves connecting $\mathcal B_w$ to $\mathcal T_w$, say
\begin{equation}
\label{eqn:pipesequence}
\mathcal B_w=\mathcal P_0\mapsto\mathcal P_1\mapsto \ldots \mapsto \mathcal P_N=\mathcal T_w.
\end{equation}
Let $n_{i,j}=\#\{k: \text{$\mathcal P_k \mapsto \mathcal P_{k+1}$ moves $+_{(i,j)}$}\}$.  By definition,  $\mathcal P_k \mapsto \mathcal P_{k+1}$ moves exactly one plus, labeled by an element of $D(w)$.  Therefore,
\begin{equation}
\label{eqn:n}
N=\sum_{(i,j)\in D(w)}n_{i,j}.
\end{equation}
\begin{claim}
\label{claim:plusrank}
If $(i,j)\in D(w)$ then 
$n_{i,j}=r_w(i,j)$.
\end{claim}
\begin{proof}
By Lemma~\ref{lemma:inherit}, $+_{(i,j)}$ must move from position  $(i,j-r_w(i,j))$ in $\mathcal B_w$ to position $(i-r_w(i,j),j)$ in $\mathcal T_w$.  At each step $+_{(i,j)}$ remains stationary or it moves up row and one column to the right.  So  $+_{(i,j)}$ must move exactly  $i-(i-r_w(i,j))=r_w(i,j)$ times to go from row $i$ to row $i-r_w(i,j)$.
\end{proof}
Then
\begin{align*}
\eta_w&=\sum_{(i,j)\in D(w)} r_w(i,j) &\text{(by Lemma~\ref{lemma:eta})}\\
&=\sum_{(i,j)\in D(w)} n_{i,j} &\text{(by Claim~\ref{claim:plusrank})}\\
&=N &\text{(by (\ref{eqn:n}))}.
\end{align*}
Each $\mathcal P_i$ in the sequence (\ref{eqn:pipesequence}) is distinct. So \[\#{\tt RP}(w)\geq N+1.\]  
Therefore 
\[
\nu_w=\#{\tt RP}(w)\geq N+1=\eta_w+1. \qedhere
\]
\end{proof}

\section*{Acknowledgments}
The author would like to thank Alexander Yong for helpful conversations and feedback.  The author was supported by the Ruth V.~Shaff and Genevie I.~Andrews Fellowship.

\bibliographystyle{mybst}
\bibliography{mylib}

\end{document}